\documentclass[12pt]{amsart}
\usepackage{graphicx}
\usepackage[linktocpage=true,colorlinks,citecolor=blue,linkcolor=blue,urlcolor=blue]{hyperref}
\usepackage{amssymb}
\usepackage{cite}
\usepackage{amsmath}
\usepackage{latexsym}
\usepackage{amscd}
\usepackage{tikz}
\usepackage{amsthm}
\usepackage{mathrsfs}
\usepackage{url}
\usepackage[utf8]{inputenc}
\usepackage[english]{babel}
\usepackage{amsfonts}
\usepackage{mathtools}

\numberwithin{equation}{section}
\def\P{\mathbb P}

\def\E{\mathbb E}

\def\l{\left}
\def\r{\right}

\def\ee{\varepsilon}

\newtheorem{theorem}{Theorem}[section]
\newtheorem{lemma}[theorem]{Lemma}
\newtheorem{proposition}[theorem]{Proposition}

\theoremstyle{remark}
\newtheorem{remark}[theorem]{Remark}

\theoremstyle{definition}

\theoremstyle{remark}

\numberwithin{equation}{section}

\begin{document}
	\title{ On a Tur\'an Conjecture and random multiplicative functions  }
	\author{Rodrigo Angelo}
	\address{Department of Mathematics, Stanford University, Stanford, CA, USA}
	\email{rsangelo@stanford.edu}
	\author{Max Wenqiang Xu}
	\address{Department of Mathematics, Stanford University, Stanford, CA, USA}
	\email{maxxu@stanford.edu}
	\begin{abstract}
We show that if $f$ is the random completely multiplicative function, the probability that $\sum_{n\le x}\frac{f(n)}{n}$ is positive for every $x$ is at least $1-10^{-45}$, while also strictly smaller than $1$. For large $x$, we prove an asymptotic upper bound of $O(\exp(-\exp( \frac{\log x}{C\log \log x })))$ on the exceptional probability that a particular truncation is negative, where $C$ is some positive constant.
\end{abstract}
	\maketitle
\section{Introduction}	Tur\'an noticed that if every truncation of the sum
	$$\sum_{n\ge 1} \frac{\lambda(n)}{n}$$
	is positive, where $\lambda $ is the Liouville function, the Riemann hypothesis would follow \cite{turan}. This conjecture\footnote{While the literature refers to this as T\'uran's conjecture, strictly speaking he never conjectured this, only remarked on the relationship with the Riemann hypothesis.} was first disproved by Haselgrove \cite{Has}, and eventually Borwein, Ferguson, and Mossinghoff \cite{BFM} found that 
\begin{equation}\label{eqn: N_0}
    N_0 = 72, 185, 376, 951, 205
\end{equation} is the smallest integer $x$ such that $\sum_{n\le x} \lambda(n)/n$ is negative.  See \cite{GS07} for more discussion.
	
	Tur\'an's conjecture inspired our study of the positivity of these sums for a random completely multiplicative function instead. See similar flavor problems studied in \cite{aymone, kalm}. The random completely multiplicative function is defined to be $f(p) = \pm 1$ with probabilities $\frac{1}{2}$ independently at each prime, and it is extended completely multiplicatively to all natural numbers.
	We are then interested in the probability of the event that for every $x$ the partial sum
	\begin{equation}\label{eqn: partial sum}
	   \sum_{n\le x} \frac{f(n)}{n}    
	\end{equation}
is positive. This probability is at most $1-2^{-\pi(N_0)}$ ($\pi(N_0)$ denotes the number of primes up to $N_0$), since if $f$ matches $\lambda$ at each $n \le N_0$, then the truncation at $x= N_0$ fails to be positive. We prove however that this probability is still very close to $1$:
\begin{theorem}\label{thm: numerical}
	Let $f$ be a sample of the random completely multiplicative function. The probability that
		$$\sum_{n\le x} \frac{f(n)}{n} $$
		is positive for every $x$ is at least $1-10^{-45}$.
	\end{theorem}

We also study the probability that a particular truncation is negative for large $x$, obtaining an asymptotic result.


\begin{theorem}\label{thm: truncate}
		Let $f$ be a sample of the random completely multiplicative function. There exists a constant $C>0$ such that for large $x$, the probability that
		$$\sum_{n\le x} \frac{f(n)}{n} > 0$$
		is at least $1- O(\exp(-\exp( \frac{\log x}{C\log \log x })))$.
	\end{theorem}
\begin{remark}
    It is proved in \cite{GS07} that for each large $x$ there is at least one completely multiplicative function for which $\sum_{n\le x}\frac{f(n)}{n} <0$. Hence this bound cannot be improved further than $1-2^{-\pi(x)}$.
\end{remark}

\subsection*{Notation}  We write $f\ll g$ or $f =O(g)$ if there exists a positive constant $C$ such that $f\le C g$, and $f=o(g)$ if $f(x) \le \ee g(x)$ for any $\ee>0$ when $x$ is sufficiently large. We use $f * g$ to denote the Dirichlet convolution of $f$ and $g$. In this paper, we use $C, C_i$ to denote positive constants for which the exact values are not important.

\subsection*{Acknowledgement}
We are grateful to Kannan Soundararajan for helpful discussions about the problem and comments on earlier versions of this paper, and also for his excellent lectures at Stanford University on topics in number theory which provided useful background knowledge 
for this project. We thank the anonymous referee for helpful corrections and many useful suggestions, in particular, for pointing out a computational mistake at the end of the earlier version of the paper.

\section{ Truncations at all values: Proof of Theorem~\ref{thm: numerical}}
In this section, we prove a numerical bound for the probability that the random sum $\sum_{n\le x} \frac{f(n)}{n}$ is always positive.

Our proof begins with the observation that the truncations at $x\le N_0$ are always positive, not just for the Liouville function but for any completely multiplicative function (Proposition~\ref{lem: N_0}). This step is crucial as it allows us to ignore small $x$.

We then factor the random sum $\sum_{n \le x} \frac{f(n)}{n}$ as a random Euler product $\prod_{p \le x} (1-\frac{f(p)}{p})^{-1}$ plus a tail error. We prove that the original sum is positive with high probability by showing that the random product is positive and bounded away from $0$ with high probability (Proposition~\ref{prop: large product}) while the tail error is very small with high probability (Proposition~\ref{prop: large tail}). The bound on the random product uses a Chernoff type bound (Lemma~\ref{lem:euler product large}) together with an application of Etemadi's inequality on the maximum of a random walk (Lemma~\ref{lem: eteapplication}). The bound on the tail uses the method of moments (Lemma~\ref{truncationprobability}).
Along the way, we define various arithmetic sums and products to be computed numerically, which is needed to get an efficient final numeric bound. Our proof of Theorem \ref{thm: truncate} will follow the same outline, except we plug in asymptotic bounds for these sums and products.

	
\begin{proposition}\label{lem: N_0}
    Let $f$ be a completely multiplicative function with values in $\{1,-1\}$, and $N_0$ be defined as in \eqref{eqn: N_0}. Then $\sum_{n\le x} \frac{f(n)}{n}> 0$ for any $x < N_0$.
\end{proposition}

\begin{proof}
Consider $g = f * |\mu|$, where $\mu$ denotes the Möbius function. The value of this function at powers of primes is either 0 or 2, hence it is nonnegative. And because $|\mu| * \lambda $ is the identity of Dirichlet convolution, we have $f = g * \lambda$.

By expanding $f(n) = \sum_{ml=n} g(m)\lambda(l)$ we obtain
    \[ \sum_{n\le x} \frac{f(n)}{n} =  \sum_{m\le x} \frac{g(m)}{m} \sum_{\ell \le x/m} \frac{\lambda(\ell)}{\ell} . \]
Since $\frac{x}{m} \le x < N_0$, the Liouville function sums are all positive. Combining this with the fact that $g$ is non-negative and $g(1)=1$, we conclude that $\sum_{n\le x}\frac{f(n)}{n}>0$.    
\end{proof}

We now deal with $x \ge N_0$. We write the partial sum $\sum_{n\le x} \frac{f(n)}{n}$ as an Euler product plus a tail term. Let $P(n)$ denote the largest prime factor of $n$, then   
\begin{equation}\label{eqn: decomp}
    \sum_{n \le x} \frac{f(n)}{n} = \prod_{p\le x}\l(1-\frac{f(p)}{p}\r)^{-1} - \sum_{\substack{n > x\\ P(n)\le x}} \frac{f(n)}{n}.
\end{equation}
We show that the random Euler product is likely to be reasonably large and the tail is typically small. The following two propositions together imply Theorem~\ref{thm: numerical}.	

\begin{proposition}\label{prop: large product}
Let $f$ be a sample of random completely multiplicative function. Then \[
 \P \l( \prod_{p\le x}\l(1-\frac{f(p)}{p}\r)^{-1} \le   0.12 ~\text{for some $x\ge N_0$}\r)  \le 5\cdot 10^{-46}.  \]
\end{proposition}

\begin{proposition}\label{prop: large tail}
Let $f$ be a sample of random completely multiplicative function. Then \[
 \P \l(  \sum_{\substack{n > x\\ P(n)\le x}} \frac{f(n)}{n} >  0.12 ~\text{for some $x\ge N_0$}\r)  \le 5\cdot 10^{-46}.  \]
\end{proposition}

\begin{proof}[Proof of Theorem~\ref{thm: numerical}]
When $x\le N_0$, the partial sum is always positive by Proposition~\ref{lem: N_0}. Theorem~\ref{thm: numerical} then follows directly from Proposition~\ref{prop: large product} and Proposition~\ref{prop: large tail}.
	\end{proof}

\subsection{ Large Random Euler product}
In this subsection, we prove Proposition~\ref{prop: large product}.

We begin by showing that the random Euler product is likely to be bounded below from zero with large probability for any given $x$. 
	\begin{lemma}\label{lem:euler product large}
	Let $f$ be the random completely multiplicative function. Let $x>0$ and $\delta >  0$. Then for a parameter $\lambda \le \frac{x}{10}$ to be optimized one has:
\begin{equation}\label{eqn: delta 4}
    \P \l( \prod_{p\le x}\l(1-\frac{f(p)}{p}\r)^{-1} \le   \delta \r) \le \exp \l(\log P_{\lambda} - \lambda \log \frac{1}{\delta}+ 0.05 \lambda \r),
\end{equation}
where
$$P_\lambda : = \prod_{p \le 10 \lambda} \l(\frac{(1+1/p)^{\lambda} +(1-1/p)^{\lambda}}{2} \r).$$
	\end{lemma}

\begin{remark}
In the regime that $x, \lambda $ is large, $\delta$ is small, one can deduce asymptotically good bounds on $P_\lambda$ and further bounds on the exceptional probability in the form 
$$\P \l( \prod_{p\le x}\l(1-\frac{f(p)}{p}\r)^{-1} \le   \delta \r) \le \exp(- \exp(\frac{c}{\delta}))$$ for some $c>0$.
See Lemma~\ref{lem: asymp large product} for a detailed statement. However the bound for $P_{\lambda}$ is weak for small $\lambda$, requiring a version where we can input $P_{\lambda}$ numerically. 

	
	\end{remark}
	
	\begin{proof}[Proof of Lemma~\ref{lem:euler product large}]
Let $X := \prod_{p\le x}(1-\frac{f(p)}{p})$. Then for all positive $\lambda$,
\[\E X^{\lambda} = \prod_{p\le x} \Big( \frac{(1+\frac{1}{p})^{\lambda} + (1-\frac{1}{p})^{\lambda}}{2}\Big) = P_\lambda \cdot \prod_{10\lambda< p\le x} \Big( \frac{(1+\frac{1}{p})^{\lambda} + (1-\frac{1}{p})^{\lambda}}{2}\Big) \]
This step is required because the product up to $x$ is too large to be computed numerically, so we replace it with the shorter product $P_{\lambda}$ which still approximates it well.

Notice that each factor in the above product is correspondingly
\[\le \frac{e^{\lambda /p}+e^{-\lambda /p}}{2} \le \exp\Big(\frac{\lambda^{2}}{2p^{2}}\Big).\]
This implies that 
\[\E X^{\lambda} \le P_{\lambda} \cdot \exp\Big( \sum_{10 \lambda < p \le x} \frac{\lambda^{2}}{2p^{2}} \Big)\le P_{\lambda} \cdot  \exp\Big(\frac{\lambda}{20}\Big).\]
Invoking Markov's inequality, 
$$\P(X^{-1} \le  \delta ) = \P(X \ge  \delta^{-1} ) \le \frac{\E X^{\lambda}}{\delta^{-\lambda}} $$
yields the desired result.
\end{proof}

The above Lemma~\ref{lem:euler product large} tells us that the product is likely to be large at $x = N_0$. We employ Etemadi's inequality \cite{Etemadi} to show that if the product is large at $N_0$ it is likely to remain large forever. This is achieved in Lemma~\ref{lem: eteapplication}.

\begin{lemma}[Etemadi's inequality]\label{lem: Etemadi}
Let $X_1, X_2, \dots, X_n$ be independent real-valued random variables defined on some common probability space, and let $\alpha\ge 0$. Let $S_k$ denote the partial sum
    $S_k = X_1 + \cdots+ X_k$. Then
    \[  \P \l( \max_{1\le k \le n} |S_k| \ge  3\alpha      \r) \le 3 \max_{1\le k \le n} \P \l(|S_k| \ge  \alpha  \r) .  \]
\end{lemma}

We also use the following standard inequality, which is a type of Chernoff bound (e.g. see \cite{Gut}). 

\begin{lemma}[Bernstein inequality]\label{lem: Ber}
Let $X_1, X_2, \dots, X_n$ be a sequence of independent mean zero random variables. Suppose that $|X_i|\le M$ almost surely for all $i$. Then for all positive $t$,
\[\P\Big(\sum_{i=1}^{n}X_i   \ge t \Big) \le \exp\l( -\frac{\frac{t^{2}}{2}}{\sum_{i=1}^{n} \E |X_i|^{2} + \frac{Mt}{3}} \r).\]
\end{lemma}

By using the above two lemmas, we derive the following.  
\begin{lemma}\label{lem: eteapplication}
        Let $f$ be a sample of the random completely multiplicative function. Then for any any positive integers $x$ and $ 1/2 < \ell   < 1$, 
 \[\P \l( \inf_{ y\ge 1 } \prod_{x < p\le  x+y}\l(1-\frac{f(p)}{p}\r)^{-1} \le   \ell \r) \le 3\exp \l(-\frac{x( \log (\ell^{-1}))^2}{20}\r).
 \] 
\end{lemma}
\begin{proof}
By taking the log of both sides, the event is equivalent to
\[\max_{y\ge 1} \sum_{x<p\le x+y} \log \l(1- \frac{f(p)}{p}\r)  \ge \log \ell^{-1} .\]
From the inequality $\log(1+x)<x$ for all $x$, the probability is at most
\[ \P\l(  \max_{y\ge 1} \sum_{x< p \le x+y}  \frac{f(p)}{p}      \ge \log \ell^{-1} \r). \]
  By using Etemadi's inequality, this is at most
    $$3 \max_{y \ge 1} \P\l(\sum_{x<p \le x+y} \frac{f(p)}{p} \ge \frac{\log \ell^{-1}}{3}\r).$$
 Apply  Bernstein's inequality to conclude it is 
    $$\le 3 \exp\l(-\frac{(\log \ell^{-1})^2/18}{\sum_{x<p \le x+y}\frac{1}{p^2} + \frac{\log \ell^{-1}}{9x}}\r) \le 3 \exp\l(-\frac{(\log \ell^{-1})^2/18}{\frac{1}{x} + \frac{\log \ell^{-1}}{9x}}\r) \le 3\exp \l(-\frac{x (\log \ell^{-1})^2}{20}\r)$$
where in the last step we used $\ell > \frac{1}{2}$. 
 \end{proof}

\begin{proof}[Proof of Proposition~\ref{prop: large product}]

Combining the above two lemmas, we have for any $\lambda\le N_0/10$,
\begin{equation}\label{eqn: delta}
\begin{split}
& \P \l( \prod_{p\le x}\l(1-\frac{f(p)}{p}\r)^{-1} \le   \delta ~\text{for some $x$}\r)\\ &\le
\P \l( \prod_{p\le N_0}\l(1-\frac{f(p)}{p}\r)^{-1} \le   \delta/\ell \r)   + \P \l( \inf_{ y\ge 1 } \prod_{N_0 < p\le  N_0+y}\l(1-\frac{f(p)}{p}\r)^{-1} \le   \ell \r)\\
& \le \exp \l(\log P_{\lambda} - \lambda \log \frac{\ell}{\delta}+ 0.05 \lambda \r) +  3\exp \l(-\frac{N_0( \log \ell^{-1})^2}{20}\r).
\end{split}
\end{equation}
By choosing $\ell = 0.9999$, the above bound is at most
\begin{equation}\label{eqn: c_delta}
   \exp \l(\log P_{\lambda} - \lambda \log \frac{1}{\delta}+ 0.0502 \lambda \r) +  3\exp \l(-7\cdot 10^{3}  \r). 
\end{equation}
To minimize the first quantity, we choose
$\lambda = 700, \delta = 0.12$ and the conclusion follows. See Section~\ref{subsection: num} for the discussion on parameter choice.
\end{proof}

\subsection{Small tail}
In this section, we prove Proposition~\ref{prop: large tail}.
The proof is based on moments computation. 
Let 
\begin{equation}\label{eqn: S_x}
    S_x: = \sum_{\substack{n\ge x\\  P(n)<x}} \frac{f(n)}{n}.
\end{equation}
We first show that at any given truncation point, the exceptional probability that the tail being large is small.
\begin{lemma}\label{truncationprobability}
Let $S_x$ be defined as above, $\sigma \in \mathbb{R}$ and $k$ be a positive integer. Define 
\begin{equation}\label{eqn: S(r,k,sigma)}
 S(R, k, \sigma) : = \prod_{p \le R} \left(\frac{(1- \frac{1}{p^{\sigma/2}})^{-2k} + (1+ \frac{1}{p^{\sigma/2}})^{-2k}}{2}\right) \prod_{p>R}\left(1-\frac{1}{p^{\sigma}}\right)^{-(2k^2+2k)}.    
\end{equation}
Then for any $\delta >0, \sigma>1$ and $R$ such that $2R^{\sigma} > k^2$, one has
\begin{equation}\label{eqn: con S(r,k,sigma)}
		\begin{split}
		  \P(S_x \ge \delta ) & \le  \frac{S(R, k, \sigma)}{\delta^{2k}x^{k(2-\sigma)}}.
		\end{split}
		\end{equation}
	\end{lemma}

\begin{proof}
Let $d_{2k}(n)$ be the generalized divisor of $n$, i.e. counting the number of ways writing $n$ as a product of $2k$ positive integers. Then
\[\E(S_x^{2k}) = \sum_{\substack{m^{2} = a_1a_2\cdots a_{2k} \\ a_i>x \\ P(a_i)<x} } \frac{1}{m^{2}} \le \sum_{\substack{m^{2}>x^{2k} \\ P(m)<x }} \frac{d_{2k}(m^{2})}{m^{2}}.    
		\]
To bound the right hand side, we pull out one factor of $m^{2-\sigma}$ and then write the rest as an Euler product by dropping the condition $m >x^{k}$. It follows that 
\begin{equation}\label{eqn: 2k moments}
\begin{split}\E(S_x^{2k})  
&    \le \frac{1}{x^{k(2-\sigma)}} \sum_{P(m)<x} \frac{d_{2k}(m^{2})}{m^{\sigma}} \\
& = \frac{1}{x^{k(2-\sigma)}} \prod_{p < x} \left(\frac{(1- \frac{1}{p^{\sigma/2}})^{-2k} + (1+ \frac{1}{p^{\sigma/2}})^{-2k}}{2}\right).
		\end{split}    
\end{equation}
Let $R>0$ be a threshold we use to distinguish large and small primes (again, this is necessary because a numerical computation up to $x$ terms would be impossible). For $p \le R$, we keep the terms as they are. For $p > R$, we claim that factors in the product in \eqref{eqn: 2k moments} satisfying
\begin{equation}\label{eqn: claim}
    \frac{(1- \frac{1}{p^{\sigma/2}})^{-2k} + (1+ \frac{1}{p^{\sigma/2}})^{-2k}}{2} \le \left(1-\frac{1}{p^{\sigma}}\right)^{-(2k^2+2k)}.
\end{equation}
We start with the inequality $1+x > e^{x-x^2}$ which holds for $|x|<1/2$, yielding that for all $p>R$
	$$\frac{(1- \frac{1}{p^{\sigma/2}})^{-2k} + (1+ \frac{1}{p^{\sigma/2}})^{-2k}}{2} \le e^{\frac{2k}{p^{\sigma}}} \left(\frac{e^{\frac{2k}{p^{\sigma/2}}}+e^{-\frac{2k}{p^{\sigma/2}}}}{2} \right).$$
By using Taylor expansion, $e^x \le 1+x+\frac{x^2}{2}+\frac{x^3}{6}+\frac{x^4}{8}$ for $|x|<4$ (which holds as $2R^{\sigma}>k^2$), the above is further bounded by
		$$\le e^{\frac{2k}{p^{\sigma}}} \left(1+\frac{2k^2}{p^{\sigma}}+\frac{2k^4}{p^{2\sigma}}\right),$$
 as the odd powers cancel out. By using Newton binomial expansion, this is at most
		$$\le e^{\frac{2k}{p^{\sigma}}} \left(1-\frac{1}{p^{\sigma}}\right)^{-2k^2} \le \left(1-\frac{1}{p^{\sigma}}\right)^{-(2k^2+2k)}, $$
where the last inequality above uses $e^{y}\le (1-y)^{-1}$ for $0 < y <1$. This proves \eqref{eqn: claim}.
Collecting all bounds, one has the moment bound
\begin{equation}\label{eqn: markov}
  \E(S_x^{2k}) \le \frac{1}{x^{k(2-\sigma)}} \prod_{p \le R} \left(\frac{(1- \frac{1}{p^{\sigma/2}})^{-2k} + (1+ \frac{1}{p^{\sigma/2}})^{-2k}}{2}\right) \prod_{R< p \le x}
		\left(1-\frac{1}{p^{\sigma}}\right)^{-(2k^2+2k)}.  
\end{equation}
The desired bound \eqref{eqn: con S(r,k,sigma)} on exceptional probability follows from an application of Markov's inequality and relaxing $x$ in the last product to infinity. In our final choice of parameters the contribution of the last product in \eqref{eqn: markov} is negligible compared with the rest. 

\end{proof}

\begin{remark}
Instead of proving \eqref{eqn: claim}, one may alternatively prove $d_{2k}(p^{2m})\le d_{2k^{2}+2k}(p^{m})$ by considering induction on $m$. This may simplify some computations above. 
\end{remark}

Take the union bound to control the exceptional probability for $S_x$ to be large for some $x>N_0$. 	
	\begin{lemma}\label{lem: exceptional prob for smooth}
Let $S_x$ and $S(R,k, \sigma)$ be defined as in \eqref{eqn: S_x} and \eqref{eqn: S(r,k,sigma)} respectively, and conditions in Lemma~\ref{truncationprobability} are satisfied. Further, assume that $k(2-\sigma)>1$, then 
\begin{equation}\label{eqn: opt 2}
    \P\l( S_x > \delta ~\text{for some $x\ge N_0$}  \r) \le  \frac{1}{\delta^{2k}}\frac{S(R, k, \sigma)}{(k(2-\sigma)-1)N_0^{k(2-\sigma)-1}}.
\end{equation}

	\end{lemma}
\begin{proof}
Applying Lemma \ref{truncationprobability} for all $x \ge N_0$ and taking the union bound, the exceptional probability in \eqref{eqn: opt 2} is at most
	    $$\frac{S(R, k, \sigma)}{\delta^{2k}} \sum_{x \ge N_0} \frac{1}{x^{k(2-\sigma)}} \le \frac{1}{\delta^{2k}}\frac{S(R, k, \sigma)}{(k(2-\sigma)-1)N_0^{k(2-\sigma)-1}}.  $$
	\end{proof}

\begin{proof}[Proof of Proposition~\ref{prop: large tail}]
We choose
$\lambda = 700, \delta = 0.12, k = 48$, $\sigma = 1.42$ and $R=10^{4}$ in Lemma~\ref{lem: exceptional prob for smooth} and the conclusion follows.
\end{proof}	
	
\subsection{A discussion on choice of parameters}\label{subsection: num}	
	
We discuss the choices of our parameters above. By combining \eqref{eqn: c_delta} and \eqref{eqn: opt 2}, the task is to find $\lambda, \delta, k, \sigma$, and $R$ that minimize
\[  3\exp \l(-7\cdot 10^{3}  \r) + \exp \l(\log P_{\lambda} - \lambda \log \frac{1}{\delta}+ 0.0502 \lambda \r)+\frac{1}{\delta^{2k}}\frac{S(R, k, \sigma)}{(k(2-\sigma)-1)N_0^{k(2-\sigma)-1}}.\]
	
A key point here is that we may rewrite the product over primes bigger than $R$ in $S(R,k,\sigma)$ as
		$$\left(\frac{\zeta(\sigma)}{\prod_{p \le R} (1-\frac{1}{p^{\sigma}})^{-1}}\right)^{2k^2+2k},$$
involving the product of only primes up to $R$.
Therefore all the terms in the expression we want to minimize, particularly $P_{\lambda}$ and $S(R,k,\sigma)$ can be computed efficiently, not requiring products larger than $10\lambda$ or $R$. Increasing the value of $\lambda$ or $R$ essentially always improves the precision of the bound, so we set them to be around $10^4$ where the computation is manageable.

Once we had code for computing this expression, we optimized it with a random descent by hand, which led to our parameter choice. See the code in \cite{code}.

\section{Truncations at large values: Proof of Theorem~\ref{thm: truncate}}

In this section we give an asymptotic bound on the exceptional probability that $\sum_{n \le x} \frac{f(n)}{n}$ is non-positive, for large values of $x$. We use the same bounds as in the previous sections, except we evaluate them asymptotically.

\begin{lemma}\label{lem: asymp large product}
Let $f$ be a sample of random completely multiplicative function. Let $x>0$ be large and $(\log x)^{-1}\ll \delta $ be small. Then 
\begin{equation}\label{eqn: asymp 1}
  \P \l( \prod_{p\le x}\l(1-\frac{f(p)}{p}\r)^{-1} \le   \delta \r) \le \exp(- \exp(\frac{c}{\delta})),  
\end{equation}
for some $c>0$.
\end{lemma}	
\begin{proof}
We derive this result from Lemma~\ref{lem:euler product large}. Notice that 
$$P_{\lambda} \le \prod_{p \le 10 \lambda} \Big(1+\frac{1}{p}\Big)^{\lambda} \ll (C_0 \log \lambda)^{\lambda}$$ for some $C_0>0$. We set $ \log \lambda =  \frac{1}{2C_0 \delta}$. Then  for large $x$ and small $\delta$, \eqref{eqn: delta 4} gives 
$$\P \l( \prod_{p\le x}\l(1-\frac{f(p)}{p}\r)^{-1} \le   \delta \r) \ll \exp \l(-\frac{\lambda}{2} \r)\ll \exp(- \exp(\frac{c}{\delta})),$$ for some $c>0$.

\end{proof}

\begin{lemma}
Let $S_x$ be defined as in \eqref{eqn: S_x}, $k\ge 0$ and $3/2< \sigma <2$ to be optimized. Then for any $\delta >0,$ and $R>0$ such that $2R^{\sigma} > k^2$,
\begin{equation}\label{eqn: asymp 2}
  \P(S_x \ge \delta) \le \frac{1}{\delta^{2k} x^{k(1-\sigma/2)}} \exp \left(  16k^{2/\sigma} + 4k \sum_{p\le R} \frac{1}{p^{\sigma/2}} \right).
\end{equation}
\end{lemma}

\begin{proof}
Apply Lemma \ref{truncationprobability} to get
$$ \P(S_x \ge \delta )  \le  \frac{S(R, k, \sigma)}{\delta^{2k}x^{k(2-\sigma)}}.		  $$

Choosing the cutoff $R = k^{2/\sigma}$, let us bound $S(R, k, \sigma)$. For the primes less than $R$:
$$\prod_{p \le R} \left(\frac{(1- \frac{1}{p^{\sigma/2}})^{-2k} + (1+ \frac{1}{p^{\sigma/2}})^{-2k}}{2}\right) \le \prod_{p \le R} (1-\frac{1}{p^{\sigma/2}})^{-2k}.$$
From the inequality $1-x\ge e^{-2x}$ for small $x$, the above product is 
$$\le \exp\Big(4k \sum_{p \le R} \frac{1}{p^{\sigma/2}}\Big). $$
Meanwhile for the contribution of $p > R$ to $S(R,k,\sigma)$ we bound it similarly to get
$$\prod_{p>R} (1-\frac{1}{p^\sigma})^{-2k^2-k} \le \exp\Big(2(2k^2+k) \sum_{p > R} \frac{1}{p^{\sigma}}\Big) \le \exp\Big(16 k^2R^{1-\sigma}\Big) = \exp (16 k^{2/\sigma}).$$
This completes the proof.
\end{proof}





\begin{proof}[Proof of Theorem~\ref{thm: truncate}]
We choose $k = x^{\frac{1}{\log \log x}}, \delta = \frac{ \log \log x}{\log x}, \sigma = 2- \frac{C_1 \log \log x}{\log x} $ for suitable large constant $C_1$. The exceptional probability in \eqref{eqn: asymp 1} is clearly satisfied the bound. 
As for the exceptional probability in \eqref{eqn: asymp 2}, we bound
\[4k \sum_{p\le R} \frac{1}{p^{\sigma/2}} \le 5 k \log \log R \le  5 x^{\frac{1}{\log \log x}} \log \log x, \]
in the given range of $\sigma$. 
As $C_1$ is large enough, the right hand side of \eqref{eqn: asymp 2} is (for large $x$)
\[\le \exp(2k \log \log x- C_1 k\log\log  x/2 + 100k   ) \le \exp (-C_1k\log \log x /3),   \]
which completes the proof by plugging the choice of $k$.   
\end{proof}






\bibliographystyle{plain}
	\bibliography{turan}{}

\end{document}